\documentclass[11pt,a4paper]{amsart}
\let\amsmarkboth\markboth    
\usepackage{amsmath,amsfonts,amssymb,amsthm,amscd}
\usepackage[english]{babel}
\usepackage[mathscr]{eucal}
\usepackage{graphicx}
\usepackage{appendix}
\usepackage{color}

\makeatletter
\let\markboth\amsmarkboth
\bbl@redefine#1#2{%
   \def\bbl@arg{#1}%
   \ifx\bbl@arg\@empty
     \toks@{}%
   \else
     \toks@{\noexpand\foreignlanguage{%
              \languagename}{%
              \noexpand\bbl@restore@actives#1}}%
   \fi
   \def\bbl@arg{#2}%
   \ifx\bbl@arg\@empty
     \toks8{}%
   \else
     \toks8{\noexpand\foreignlanguage{%
              \languagename}{%
              \noexpand\bbl@restore@actives#2}}%
   \fi
   \edef\bbl@tempa{\the\toks@}%
   \edef\bbl@tempb{\the\toks8}%
   \protected@edef\bbl@tempa{%
     \noexpand\org@markboth{\bbl@tempa}{\bbl@tempb}}%
   \bbl@tempa
}
\DeclareRobustCommand*\ams@disablelinebreak{\def\\{ \ignorespaces}}
\def\maketitle{\par
   \@topnum\z@ %
   \@setcopyright
   \thispagestyle{firstpage}%
   \uppercasenonmath\shorttitle
   \ifx\@empty\shortauthors \let\shortauthors\shorttitle
   \else \andify\shortauthors
   \fi
   \@maketitle@hook
   \begingroup
   \@maketitle
   \toks@\@xp{\shortauthors}\@temptokena\@xp{\shorttitle}%
   \protected@edef\@tempa{%
     \@nx\markboth{\ams@disablelinebreak
       \@nx\MakeUppercase{\the\toks@}}{\the\@temptokena}}%
   \@tempa
   \endgroup
   \c@footnote\z@
   \@cleartopmattertags
}

\makeatother

\numberwithin{equation}{section}

\newcommand{\N}{{\mathbb N}}
\newcommand{\R}{{\mathbb R}}
\newcommand{\real}{{\mathbb R}}

\newcommand{\RR}{{\mathbb R}^2}

\newcommand{\diver}{{\rm div}\,}

\newcommand{\loc}{\operatorname{{loc}}}

\newcommand{\res}{\mathop{\hbox{\vrule height 7pt width .5pt depth 0pt
\vrule height .5pt width 6pt depth 0pt}}\nolimits}

\newtheorem{theorem}{Theorem}[section]
\newtheorem{proposition}[theorem]{Proposition}
\newtheorem{lemma}[theorem]{Lemma}

\theoremstyle{definition}
\newtheorem{definition}[theorem]{Definition}

\theoremstyle{remark}
\newtheorem{remark}[theorem]{Remark}

\newcommand{\eps}{\varepsilon}

\begin{document}


\title[Flows of vector fields with point singularities]{Flows of vector fields with point singularities and the vortex-wave system}

\author[G. Crippa]{G. Crippa}
\address[G. Crippa]{Departement Mathematik und Informatik\\ Universit\"at Basel\\ Rheinsprung 21,
CH-4051, Basel, Switzerland}
\email{gianluca.crippa@unibas.ch}

\author[M. C. Lopes Filho]{M. C. Lopes Filho}
\address[M. C. Lopes Filho]{Instituto de Matem\'atica\\
Universidade Federal do Rio de Janeiro\\
Cidade Universit\'aria -- Ilha do Fund\~ao\\
Caixa Postal 68530\\
21941-909 Rio de Janeiro, RJ -- Brasil.}
\email{mlopes@im.ufrj.br}

\author[E. Miot]{E. Miot}
\address[E. Miot]{D\'epartement de Math\'ematiques\\ B\^atiment 425\\Universit\'e Paris-Sud 11 \\ 91405 Orsay, France}  \email{evelyne.miot@math.u-psud.fr}

\author[H. J. Nussenzveig Lopes]{H. J. Nussenzveig Lopes}
\address[H. J. Nussenzveig Lopes]{Instituto de Matem\'atica\\
Universidade Federal do Rio de Janeiro\\
Cidade Universit\'aria -- Ilha do Fund\~ao\\
Caixa Postal 68530\\
21941-909 Rio de Janeiro, RJ  -- Brasil.}
\email{hlopes@im.ufrj.br}

\date{\today}

\begin{abstract}
The vortex-wave system is a version of the vorticity equation governing the motion of 2D incompressible fluids in which vorticity is split into a finite sum of Diracs, evolved
through an ODE, plus an $L^p$ part, evolved through an active scalar transport equation. Existence of a weak solution for this system 
was recently proved by Lopes Filho, Miot and Nussenzveig Lopes, for $p>2$, but their result left open the existence and basic properties
of the underlying Lagrangian flow. In this article we study existence, uniqueness and the qualitative properties of the (Lagrangian flow for the) 
linear transport problem associated to the vortex-wave system. To this end, we study the flow associated to 
a two-dimen\-sional vector field which is singular at a moving point. We first observe that existence and uniqueness of the regular Lagrangian flow are ensured by combining previous results by Ambrosio and by Lacave and Miot. In addition we prove that, generically, the Lagrangian trajectories do not collide with the point singularity. In the second part we present an approximation scheme for the flow, with explicit error estimates obtained by adapting results by Crippa and De Lellis for Sobolev vector fields.  

\end{abstract}

\maketitle

\section{Introduction}

The purpose of this article is to study the flow associated to a particular class of vector fields that contain a point singularity, which arise as weak solutions of the vortex-wave system. For a smooth vector field $b:[0,T]\times \R^2\to \R^2$, the flow of $b$ is the unique map $X:[0,T]\times \R^2\to \R^2$ defined by
\begin{equation}
\label{eq:ODE} 
\begin{cases}
  \displaystyle \frac{d}{dt} X(t,x)=b\left(t,X(t,x)\right)\quad t\in [0,T],\\
\displaystyle X(0,x)=x\in \RR.
 \end{cases}
\end{equation}
It turns out that, in some cases, even if $b$ is not smooth, it is still possible to define an extended notion
of flow for $b$, nowadays called regular Lagrangian flow (see, e.g., Definition
\ref{def:lagrangian-flow} below). In their pioneering work, DiPerna and Lions \cite{dip-lions} proved the existence and uniqueness of the flow for vector fields belonging to $L^1(W^{1,1}_{\loc})$ with suitable decay at infinity and 
with bounded divergence (see assumptions $(H_1)$ and $(H_3)$ below).  The Sobolev-type regularity assumptions on $b$ were later relaxed
 by Ambrosio \cite{ambrosio}, allowing for $BV$ vector fields (see assumption $(H_2)$). There is a wide literature devoted to this issue, see e.g. \cite{ambrosio-2, ambrosio-crippa} and references therein  for additional or related results. The problem we address here is that the vector field associated to the vortex-wave system is not $BV$. 

\medskip

We will focus on the case where $b$ is given by
\begin{equation}\label{eq:ansatz}
 \begin{split}
  b(t,x)=v(t,x)+H(t,x),
 \end{split}
\end{equation}
where the field $v$ enters the class of vector fields considered in the theory of DiPerna and Lions and Ambrosio, and  where $H$ is a special vector field which is singular along a curve in space time. 
More precisely, we assume that the first component $v$ satisfies the same assumptions as in \cite{ambrosio}:
\begin{equation*}
 \begin{split}
 & (H_1)\quad \frac{v}{1+|x|}\in L^1\left( [0,T],L^1(\RR)\right)+ L^1\left( [0,T],L^\infty(\RR)\right),\\
& (H_2)\quad   v\in L^1\left([0,T],\text{BV}_{\loc}(\RR)\right),\\
& (H_3)\quad \diver(v)\in L^1\left([0,T],L^\infty(\R^2)\right).
\end{split}
\end{equation*}
The result of Ambrosio \cite{ambrosio} ensures existence and uniqueness of the regular Lagrangian flow associated to such fields.
In addition, in our context, we require the following assumption:
\begin{equation*}
\begin{split} 
&(H_4)\quad v\in L^\infty\left([0,T],L^q(\RR)\right)\quad \text{for some } 2<q\leq +\infty.\end{split}\end{equation*}
Next, we define our singular part $H$ as follows. We consider a given Lipschitz trajectory  in $\R^2$:
\begin{equation}\label{eq:ODEvortex}
z\in W^{1,\infty}([0,T],\R^2).
\end{equation}
We introduce the map
\begin{equation*}
 \quad K:\RR\setminus\{0\}\to \RR,\quad K(x)=\frac{x^\perp}{|x|^2}=\frac{(-x_2,x_1)}{|(x_1,x_2)|^2}
\end{equation*}
and we define 
\begin{equation}\label{eq:defH}
 H(t,x)=K\left( x-z(t)\right).
\end{equation}
Then $H$ satisfies $(H_1)$ and $(H_3)$: actually, it is divergence free. It does not satisfy $(H_2)$ therefore such a field is not covered by the result of Ambrosio \cite{ambrosio}. However note that $H$ is smooth off of the 
set $\left\{(t,z(t)),\:\: t\in [0,T]\right\}$.

The structure described by \eqref{eq:ansatz} includes that of solutions of the vortex-wave system in the special case of a single vortex together with compactly supported $L^p$ vorticity, $p>1$.

\medskip

Next we recall, following DiPerna and Lions \cite{dip-lions} and Ambrosio \cite{ambrosio}, the definition of regular Lagrangian flow.  We denote by $\mathcal{L}^2$ the Lebesgue measure on $\RR$.
\begin{definition}[Regular Lagrangian flow]
 \label{def:lagrangian-flow}
We say that a map
$X : [0, T]\times \R^2\to \RR$ is a regular Lagrangian flow for the vector field $b$ if:

(i) There exists an $\mathcal{L}^2$-negligible set $S\subset \RR$ such that for all $x\in \RR\setminus S$ the map $t\mapsto b(t, X(t, x))$ belongs to $L^1([0,T])$, and  
 $$X(t,x)=x+\int_0^t b(s, X(s,x))\,ds,\quad \forall t \in [0, T].$$

(ii) For all $R>0$ there exists $L_R>0$ such that  
$$X(t,\cdot)_\#(\mathcal{L}^2\res B_R)\leq L_R\mathcal{L}^2,\quad \forall t\in[0,T],$$ i.e.  
$\mathcal{L}^2\left(X(t, \cdot)^{-1}(A)\cap B_R\right)\leq L_R\mathcal{L}^2(A)$
for every Borel set $A\subset \RR$.

\end{definition}
In Section \ref{section:proof} we combine the abstract theory by Ambrosio \cite{ambrosio-2} (Theorem \ref{thm:abstract} below), see also \cite{ambrosio-crippa}, exploiting the link between the ODE and the continuity and transport equations (see \eqref{eq:PDE}-\eqref{eq:transport}),  with an extension of a renormalization result by Lacave and Miot \cite{clacave-miot} to show existence and uniqueness for the regular Lagrangian flow of $b$. Moreover, we prove the additional property that for $\mathcal{L}^2$-a.e. $x\in \R^2$ the trajectory starting from the point $x$ does not collide with the singularity point. More precisely, we prove the following theorem:

\begin{theorem}\label{thm:main}
 Let $b$ be as in \eqref{eq:ansatz}, where $v$ satisfies $(H_1)-(H_2)-(H_3)-(H_4)$ and where $H$ is given by \eqref{eq:defH}. Then there  exists a unique regular Lagrangian flow.  Moreover, for $\mathcal{L}^2$-a.e. $x\in \R^2$ we have $$X(t,x)\neq z(t),\quad \forall
t\in [0,T].$$ 
\end{theorem}

Observe that, by the very definition of regular Lagrangian flow, the absolute continuity of the measure $X(t,\cdot)_\#\mathcal{L}^2$ with respect to $\mathcal{L}^2$ implies, by Fubini's theorem, that for $\mathcal{L}^2$-a.e. $x\in \R^2$ we have $X(t,x)\neq z(t)$ for $\mathcal{L}^1$-a.e. $t\in [0,T]$. The main point of Theorem \ref{thm:main} is that collisions between the Lagrangian trajectories and the singularity point are avoided for \emph{all} $t\in [0,T]$. Indeed Proposition \ref{lemma:bad-set} yields a quantitative control of the amount of Lagrangian trajectories getting closer than $\eps$ to the point singularity: the proof of this proposition uses the additional assumption $(H_4)$. We mention that an analogous control on the trajectories was performed in the setting of the Vlasov-Poisson equation with singular fields by Caprino, Marchioro, Miot and Pulvirenti \cite{italiens-miot}.

\medskip

In the second part of this work we present an effective construction of the regular Lagrangian flow by an approximation argument.  In contrast with the point of view adopted in the first part, this construction does not rely on the link between the ODE and the PDE. Moreover, we provide a quantitative rate of convergence, by extending to our setting the estimates performed by Crippa and De Lellis  \cite{crippa-delellis} for vector fields without singular part.  We restrict ourselves to vector fields satisfying the stronger assumptions:
\begin{equation*}
 \begin{split}
 & (H'_1)\quad v\in L^\infty( [0,T]\times \RR),\\
& (H'_2)\quad  \nabla v\in L^1\left([0,T],L^p(\RR)\right)\quad \text{for some } 1<p\leq +\infty,\\
& (H'_3)\quad \diver(v)\in L^1\left([0,T],L^\infty(\R^2)\right).
\end{split}\end{equation*} 
In Section \ref{section:approx} we define a suitable smooth approximation $(b_n)_{n\in \N}$ of $b$,
and we denote by $X_n$ the unique corresponding classical flow.  We prove the following theorem:
\begin{theorem}\label{thm:main2} Let $v$ satisfy $(H'_1)-(H'_2)-(H'_3)$. Let $R>0$. There exists $\widetilde{R}$ and $C$, depending 
on $R$, $T$, $\|v\|_{L^\infty(L^\infty)}$, $\|\diver (v)\|_{L^1(L^\infty)}$,  $\|\nabla v\|_{L^1(L^p)}$, and $\|z\|_{W^{1,\infty}}$, such that, denoting by
\begin{equation*}\begin{split}
 \delta(n,m)&= \|b_n-b_m\|_{L^1([0,T]\times B_{\widetilde{R}})}
\end{split}
\end{equation*}
the following estimate holds:
\begin{equation*}
 \int_{B_R} \sup_{t\in [0,T]} |X_n(t,x)-X_m(t,x)|\,dx\leq \frac{C}{|\ln \delta(n,m)|^{1/3}}.
\end{equation*}
In particular, $(X_{n})_{n\in \N}$ is a Cauchy sequence in $L^1_{\loc}\left(\RR,L^\infty\left([0,T]\right)\right)$, hence for $\mathcal{L}^2$-a.e. $x\in \RR$ the sequence $(X_n(\cdot,x))_{n\in \N}$ converges uniformly to $X(\cdot,x)$ on $[0,T]$, where $X$ is the regular Lagrangian flow relative to $b$ as in Theorem \ref{thm:main}. 
\end{theorem}

\medskip

To conclude this introduction, we describe the vortex-wave system and the connection between the present work and this system.
 In two-dimensional incompressible fluids,  we consider a flow with initial vorticity consisting of the superposition of a diffuse part
$\omega_0 \in L^p$ for some $p \geq 1$ and a point vortex located at $z_0 \in \real^2$, with unit strength. 
The evolution of vorticity can be described by a system of equations called the vortex-wave system (with one single point vortex):
\begin{equation}
\label{eq:vortex-wave}
\begin{cases}
\displaystyle \partial_t \omega+(v+H)\cdot \nabla \omega =0\\
\displaystyle v=\frac{1}{2\pi}K\ast \omega\\
\displaystyle H(t,x)=\frac{1}{2\pi}K(x-z(t))\\
\displaystyle \dot{z}(t)=v(t,z(t)).
\end{cases}
\end{equation} 

This system was introduced by Marchioro and Pulvirenti \cite{mar_pul,livrejaune}. There are two natural notions of weak solution for this system, 
one is a solution in the sense of distributions, called {\it Eulerian solution}, while the other is a solution for which the diffuse part of the vorticity 
is constant along the trajectories of the flow, called {\it Lagrangian solution}, see \cite{bresiliens-miot,clacave-miot} for precise definitions.  (By
`trajectories of the flow' we mean the flow associated to the vector field $b = v + H$ above.) In \cite{mar_pul,livrejaune},
Marchioro and Pulvirenti established global existence of a Lagrangian solution with $\omega\in L^\infty(L^1\cap L^\infty)$. In \cite{bresiliens-miot},
Lopes Filho, Miot and Nussenzveig Lopes established global existence of an Eulerian solution with vorticity $\omega\in L^\infty(L^1\cap L^p)$, with $p>2$. For any $p>2$ Lagrangian solutions to the vortex-wave system are Eulerian. The converse was left open in \cite{bresiliens-miot}. The issue of the Lagrangian formulation is the natural requirement that flow trajectories should not collide with the point vortex. When $p=+\infty$, almost-Lipschitz regularity for the velocity $(1/2\pi)K\ast \omega$ enables to define flow trajectories in the classical sense, which do not intersect with the point vortex, starting from \emph{any} $x\neq z_0$ \cite{mar_pul,livrejaune}. For $p<+\infty$ this property is unclear.  In Section \ref{section:vortex-wave}, we use the results established in Sections \ref{section:proof} and \ref{section:approx} to show that any Eulerian solution with $\omega\in L^\infty(L^1\cap L^p)$, $p>2$, gives rise to a regular Lagrangian flow such that $\omega$ is constant along the flow trajectories, which do not, generically, collide with the point vortex. As it happens, when $p>2$, the assumptions $(H'_1)-(H'_2)-(H'_3)$ are all satisfied and, in particular, the point vortex trajectory $t\mapsto z(t)$ is Lipschitz.

\section{Proof of Theorem \ref{thm:main}}

\label{section:proof}

In the theory of DiPerna and Lions and Ambrosio \cite{ambrosio, ambrosio-2,dip-lions}, the existence, uniqueness and the stability properties of the flow associated to a field $b$ are linked to the well-posedness of the corresponding continuity equation
\begin{equation}
 \label{eq:PDE} 
\partial_t u+\diver(b u)=0\quad\text{on }(0,T)\times \RR,\quad u(0)=u_0
\end{equation}
and transport equation
\begin{equation}
 \label{eq:transport} 
\partial_t u+b \cdot \nabla u=0\quad\text{on }(0,T)\times \RR,\quad u(0)=u_0.
\end{equation}
Note that one passes formally from the ODE to the continuity and transport equations by noticing that if $X$ solves \eqref{eq:ODE} then $X(t,\cdot)_\# u_0$ solves \eqref{eq:PDE}, and $u_0\circ X(t,\cdot)^{-1}$ solves \eqref{eq:transport}. In the non-smooth case, we consider distributional solutions to \eqref{eq:PDE} and \eqref{eq:transport}.  Such distributional formulations make sense as soon as $bu$ and $u\,\diver(b)$ belong to $L_{\loc}^1$.

As a matter of fact, we have the following general abstract result due to Ambrosio \cite{ambrosio-2}, somewhat extending this connection to the non smooth context:
\begin{theorem}[Ambrosio \cite{ambrosio-2}, Theorems 3.3 and 3.5]
\label{thm:abstract} Let $b$ be a given vector field in $L^1_{\loc}([0,T]\times \RR)$. If existence and uniqueness for  \eqref{eq:PDE} hold in  $L^\infty\left(L^1\cap L^\infty\right)$ then the regular Lagrangian flow of $b$ exists and is unique.
\end{theorem}
And, besides, existence and uniqueness for \eqref{eq:PDE} hold for vector fields satisfying the assumptions $(H_1)-(H_2)-(H_3)$ or $(H'_1)-(H'_2)-(H'_3)$ \cite{ambrosio,dip-lions}.

Now, in the case where $b$ is given by \eqref{eq:ansatz}, the PDE well-posedness results cannot be applied directly because of the singular field $H$. However, the following holds:

\begin{proposition}
\label{prop:uniqueness} Let $b$ be given by \eqref{eq:ansatz}.

\noindent (1) Let $v$ satisfy the assumptions $(H_1)-(H_2)-(H_3)$. Let $u_0\in L^1\cap L^\infty$. Then \eqref{eq:PDE} has a unique solution $u\in L^\infty\left(L^1\cap L^\infty\right)$.

\noindent (2) Let $v$ satisfy the assumptions $(H'_1)-(H'_2)-(H'_3)$. Let $u_0\in L^1\cap L^r$, with $r>2$. Then \eqref{eq:PDE} has a unique solution $u\in L^\infty\left(L^1\cap L^r\right)$.

\end{proposition}

\begin{proof}

First, existence of a distributional solution follows in both cases from standard regularization arguments.

The argument for uniqueness is strictly analogous to  the one of Lacave and Miot \cite{clacave-miot}. We give the main lines for the reader's convenience. First, using the by now standard methods introduced in \cite{ambrosio,dip-lions}, it suffices to show that any solution $u$ satisfies the renormalization property:
\begin{equation}\label{eq:renorm}
\partial_t |u|+\diver((v+H)|u|)=0.
\end{equation}
The arguments leading to \eqref{eq:renorm} put together computations as in \cite{dip-lions} for case (2), and \cite{ambrosio} for case (1), with the renormalization property for the single vector field $H$ in both cases, i.e.
\begin{equation}\label{eq:renorm-2}
\partial_t |u|+\diver(H|u|)=0\quad \text{on } (0,T)\times \RR.
\end{equation}
As in \cite{clacave-miot}, we first observe that \eqref{eq:renorm-2} holds on the complement of the set $\{(t,z(t)),\: t\in [0,T]\}$. We next establish \eqref{eq:renorm-2} in $(0,T)\times \RR$. Let $\varphi\in C_c^\infty((0,T)\times \RR)$ and let $\chi \in C^\infty(\RR)$ be a radial function such that $0\leq \chi\leq 1$, $\chi=0$ on $B_{1/2}$ and $\chi=1$ on $B_1^c$. For $\eps>0$ we set $\chi_\eps(t,x)=\chi((x-z(t))/\eps)$ and $\varphi_\eps=\varphi\chi_\eps$. Since $\varphi_\eps$ is compactly supported off the set  $\{(t,z(t)),\: t\in [0,T]\}$ we know that
\begin{equation*}\begin{split}
\iint |u|\chi_\eps
&(\partial_t \varphi+ H\cdot \nabla \varphi)\,dx\,dt\\
&+\iint |u|\varphi
\left(-\frac{\dot{z}(t)}{\eps}\cdot \nabla \chi\left(\frac{x-z(t)}{\eps}\right)+H(t,x)\cdot \nabla \chi_\eps(t,x)\right)\,dx\,dt=0.
\end{split}
\end{equation*}
We remark that $H\cdot \nabla \chi_\eps=0$. Therefore in view of the bound on $\dot{z}$ the second term vanishes when $\eps$ tend to $0$ and we finally obtain \eqref{eq:renorm-2} by applying Lebesgue's theorem to the first term.
\end{proof}

Combining Part (1) of Proposition \ref{prop:uniqueness} and Theorem \ref{thm:abstract} we obtain the existence and uniqueness of the regular Lagrangian flow $X$ in Theorem \ref{thm:main}.
Therefore we only have to prove that for $\mathcal{L}^2$-a.e. $x\in \RR$ no collision between $X(t,x)$ and $z(t)$ occurs on $[0,T]$. This is a direct consequence of the  following
\begin{proposition}
 \label{lemma:bad-set}
For $0<\eps<1$ and $R>0$, let
\begin{equation*}
 P(\eps,R)=\left\{x\in B_R\setminus S\quad \text{s.t.}\quad \min_{t\in [0,T]}|X(t,x)-z(t)|<\eps\right \},
\end{equation*}
where $S$ is as in Definition \ref{def:lagrangian-flow}.
Then 
\begin{equation*}
  \mathcal{L}^2(P(\eps,R))\leq C(T,L_R,\|v\|_{L^\infty(L^q)}+\|\dot{z}\|_{L^\infty}) \eps^{1-\frac{2}{q}}.
\end{equation*}
\end{proposition}

\begin{proof}
 We adapt the strategy introduced in \cite{italiens-miot} for the Vlasov-Poisson equation.
Here, we set $\alpha=1-2/q>0$. We introduce
$$
 \Delta T =\lambda \eps^\beta,$$
where $0<\lambda<1$ is a parameter to be determined later, and where $$\beta=\frac{1+\alpha/q}{1-1/q}\geq 1.$$
We set
 $$ N=\left[\frac{T}{\Delta T}\right]-1$$
and we define
\begin{equation*}
 t_i=i\Delta T,\quad i=0,\ldots,N,\quad t_{N+1}=T,
\end{equation*}
so that
\begin{equation*}
[0,T]=\bigcup_{i=0}^N[t_i,t_{i+1}] \quad \text{with }\:|t_{i+1}-t_i|\leq \Delta T,\quad \forall i=0,\ldots,N.
\end{equation*}

We first consider the case $2<q<+\infty$.
We set 
\begin{equation*}
A=\left\{x\in B_R\setminus S:\quad\int_{t_i}^{t_{i+1}}|v(s,X(s,x))|^q\,ds\leq \eps^{-\alpha}, \quad \forall i\in\{0,\ldots,N\}\right\}
\end{equation*}
and for $i\in\{0,\ldots,N\}$ we set
\begin{equation*}
B_i=\left\{x\in B_R\setminus S:\quad\int_{t_i}^{t_{i+1}}|v(s,X(s,x))|^q\,ds\geq \eps^{-\alpha}\right\}.
\end{equation*}
By Chebyshev's inequality and Fubini's theorem we have
\begin{equation*}\begin{split}
\mathcal{L}^2(B_i)&\leq \eps^{\alpha}\int_{B_R}\int_{t_i}^{t_{i+1}}|v(s,X(s,x))|^q\,ds\,dx\\
&= \eps^{\alpha}\int_{t_i}^{t_{i+1}}\int_{B_R}|v(s,X(s,x))|^q\,dx\,ds.
\end{split}
\end{equation*}Using Property (ii) in Definition \ref{def:lagrangian-flow} for $X(s,\cdot)$ and $(H_4)$ we get
\begin{equation*}\begin{split}
\mathcal{L}^2(B_i)&\leq L_R  \eps^{\alpha}\|v\|_{L^\infty(L^q)}(t_{i+1}-t_i).
\end{split}
\end{equation*}
Therefore
\begin{equation}
\label{ineq:large-velocity}
\begin{split}
\mathcal{L}^2\left(\bigcup_{i=0}^N B_i\right)&\leq L_R T \|v\|_{L^\infty(L^q)} \eps^{\alpha}.
\end{split}
\end{equation}

Then, let  $x\in P(\eps,R)\cap A$ and let $s_0\in [0,T]$ such that $$\left|X\left(s_0,x\right)-z(s_0)\right|<\eps.$$ We can assume that $s_0\in (t_i,t_{i+1})$ for some $i\in \{0,\ldots,N\}$. Let $s_1\leq t_{i+1}$ maximal such that $\left|X\left(t,x\right)-z(t)\right|<2\eps$ on $[s_0,s_1)$. If $s_1=t_{i+1}$ then $x\in X(t_{i+1},\cdot)^{-1}\big( B(z(t_{i+1}),2\eps)\big)\cap B_R$. We assume then that $s_1<t_{i+1}$.  For $\mathcal{L}^1$-a.e. $t\in [s_0,s_1)$ we have $\dot{X}(t,x)=b(t,X(t,x))$. Now we observe that, 
even though $b$ is not uniformly bounded, the modulus $|X(t,x)-z(t)|$ is
H\"older continuous in time on the set $A$. Indeed, for $\mathcal{L}^1$-a.e. $t\in [s_0,s_1)$ such that $X(t,x)\neq z(t)$ we get, using that 
$$K(y)\cdot y=0,\quad \forall y\in \RR\setminus\{0\},$$
\begin{equation*}
\begin{split}
\frac{d}{dt}|X(t,x)-z(t)|
&=
\frac{X(t,x)-z(t)}{|X(t,x)-z(t)|}\cdot 
\big( v(t,X(t,x))-\dot{z}(t)\big),\end{split}\end{equation*}
hence
\begin{equation}
\label{ineq:bounded-derivative-flow}
\begin{split}
 \left|\frac{d}{dt}|X(t,x)-z(t)|\right|
&\leq |v(t,X(t,x))|+|\dot{z}(t)|.
\end{split}
\end{equation} 
Hence for all $t\in[s_0,s_1]$ we have by H\"older inequality
\begin{equation*}
\begin{split}
|X&(t,x)-z(t)|=|X(s_0,x)-z(s_0)|+\int_{s_0}^t \frac{d}{ds}|X(s,x)-z(s)|\,ds\\
&<\eps +\int_{s_0}^t|v(s,X(s,x))|\,ds+\int_{s_0}^t|\dot{z}(s)|\,ds\\
&\leq\eps+(t_{i+1}-t_i)^{1-\frac{1}{q}}\left(\int_{t_i}^{t_{i+1}}|v(s,X(s,x))|^q\,ds\right)^{\frac{1}{q}}+\|\dot{z}\|_{L^\infty}(t_{i+1}-t_i).\end{split}\end{equation*}

Finally, by definition of the set $A$ and by definition of $\beta\geq 1$ we get
\begin{equation*}
\begin{split}
|X(t,x)-z(t)|
&\leq\eps+\lambda^{1-\frac{1}{q}}\eps^{\beta(1-\frac{1}{q})-\frac{\alpha}{q}}+\|\dot{z}\|_{L^\infty}\lambda \eps^{\beta}\\
&\leq \eps+\lambda^{1-\frac{1}{q}}\eps+\|\dot{z}\|_{L^\infty}\lambda \eps.
\end{split}\end{equation*}
Now we choose $\lambda$ so that
$$\lambda^{1-\frac{1}{q}}+\|\dot{z}\|_{L^\infty}\lambda<1.$$
For this choice of $\lambda$ we obtain $|X(s_1,x)-z(s_1)|<2\eps$, which contradicts the definition of $s_1$ and shows that we must have $s_1=t_{i+1}$. It follows that 
$$A\cap P(\eps,R)\subset \bigcup_{i=0}^N X(t_{i+1},\cdot)^{-1}\big(B(z(t_{i+1}),2\eps)\big)\cap B_R.$$
Therefore in view of (ii) in Definition \ref{def:lagrangian-flow}, 
\begin{equation*} 
\begin{split}
\mathcal{L}^2(A\cap P(\eps,R))&\leq  \sum_{i=0}^{N} 
\mathcal{L}^2
\Big(X(t_{i+1},\cdot)^{-1}\big(B(z(t_{i+1}),2\eps)\big)\cap B_R\Big)\\
&\leq (N+1)L_R (4\pi \eps^2)
\end{split}
\end{equation*}
and finally
\begin{equation} \label{ineq:small-velocity}
\begin{split}
\mathcal{L}^2(A\cap P(\eps,R))\leq 4\pi L_R T\lambda^{-1}\eps^{2-\beta}.
\end{split}
\end{equation}

Combining \eqref{ineq:large-velocity}, \eqref{ineq:small-velocity} and using the definition of $\lambda$ we obtain
\begin{equation*}
\mathcal{L}^2(P(\eps,R))\leq C(T,L_R,\|v\|_{L^\infty(L^q)},\|\dot{z}\|_{L^\infty}) (\eps^\alpha+\eps^{2-\beta}).
\end{equation*}Since $2-\beta=\alpha$, this yields the conclusion.

\medskip

We now study the case where $q=\infty$, which is easier and does not require to introduce the sets $B_i$ and $A$. Indeed, let $x\in P(\eps,R)$. Coming back to \eqref{ineq:bounded-derivative-flow} and proceeding similarly as before we obtain for $s\in [s_0,s_1)$
\begin{equation*}
\begin{split}
|X(t,x)-z(t)|
&\leq \eps+\lambda\eps{}(\|v\|_{L^\infty(L^\infty)}+\|\dot{z}\|_{L^\infty})<2\eps
\end{split}\end{equation*}
provided that 
$$\lambda\eps{}(\|v\|_{L^\infty(L^\infty)}+\|\dot{z}\|_{L^\infty})<1.$$ This shows that 
$$P(\eps,R)\subset \bigcup_{i=0}^N X(t_{i+1},\cdot)^{-1}\big(B(z(t_{i+1}),2\eps)\big)\cap B_R,$$
and the conclusion then follows as before.

\end{proof}

\section{Proof of Theorem \ref{thm:main2}}
\label{section:approx}

We start by defining the smooth approximation involved in Theorem \ref{thm:main2}.
Let $(\rho_n)_{n\in \N}$ be the usual sequence of Friedrichs mollifyers. Let
$v_n=\rho_n\ast v$ and let
\begin{equation*}
 K_n(x)=\frac{x^\perp}{|x|^2+\frac{1}{n^2}},\quad x\in \RR,
\end{equation*}
which defines a globally bounded, divergence free and smooth vector field on $\RR$. We finally set
$$b_n(t,x)=v_n(t,x)+K_n(x-z(t)).$$

We first remark that $(|X_n(t,x)-z(t)|)_{n\in \N}$ is uniformly Lipschitz in time even though $b_n\notin L^\infty$. Indeed, by the same computation leading to
 \eqref{ineq:bounded-derivative-flow}, using that $K_n(y)\cdot y=0$ and $(H'_1)$,  we have
\begin{equation}
\label{ineq:derivative-flow-2}
\begin{split}
 \left|\frac{d}{dt}|X_n(t,x)-z(t)|\right|
&\leq |v_n(t,X_n(t,x))|+|\dot{z}(t)|\leq \|v\|_{L^\infty(L^\infty)}+\|\dot{z}\|_{L^\infty}.
\end{split}
\end{equation} In particular, we have the local equiboundedness property
\begin{equation}
 \label{ineq:bounded-flow}
\left\|X_n\right\|_{L^\infty\left([0,T]\times B_R\right)}\leq R+2\|z\|_{L^\infty}+(\|v\|_{L^\infty(L^\infty)}+\|\dot{z}\|_{L^\infty})T.
\end{equation}
On the other hand,  since $\diver(b_n)=\rho_n\ast \diver(v)$ we infer from $(H'_3)$ that 
\begin{equation}\label{ineq:jacobian1}
\sup_{n\geq 0} \int_0^T\|\diver(b_n)(s)\|_{L^\infty}\,ds\leq L_0<\infty.\end{equation}  In particular it follows from the standard theory on Jacobians that 
\begin{equation}\label{ineq:liouville} X_n(t,\cdot)_\#\mathcal{L}^2\leq e^{L_0}\mathcal{L}^2,\quad \forall t\in [0,T].\end{equation}

\medskip
Part of our subsequent analysis is borrowed from \cite{crippa-delellis}: we introduce
\begin{equation*}
 \widetilde{R}=R+2\|z\|_{L^\infty}+(\|v\|_{L^\infty(L^\infty)}+\|\dot{z}\|_{L^\infty})T
\end{equation*}
and
\begin{equation*}\begin{split}
 \delta(n,m)&= \|b_n-b_m\|_{L^1([0,T]\times B_{\widetilde{R}})}.
\end{split}
\end{equation*}
We consider the positive quantity
\begin{equation}\label{def:g}
 g_{n,m}=\int_{B_R}\sup_{t\in [0,T]}\ln\left( \frac{|X_n(t,x)-X_m(t,x)|}{\delta(n,m)}+1\right)\,dx.
\end{equation}
\begin{lemma}
 \label{lemma:ineq-g}
We have
\begin{equation*}
 g_{m,n}\leq C |\ln\delta(n,m)|^{2/3}
\end{equation*}
where $C$ depends only on $R$, $T$, $L_0$, $\|v\|_{L^\infty(L^\infty)}$, 
$\|\dot{z}\|_{L^\infty}$, and $\|\nabla v\|_{L^1(L^p)}$.
\end{lemma}

\medskip
From now on  $C$ will denote a positive constant depending only on $R$, $T$, $L_0$, $\|v\|_{L^\infty(L^\infty)}$, $\|\dot{z}\|_{L^\infty}$, and $\|\nabla v\|_{L^1(L^p)}$.

\medskip

Before proving Lemma \ref{lemma:ineq-g} we show how it implies Theorem \ref{thm:main2}.  In the following we will sometimes write $\delta$ instead of $\delta(m,n)$.

\medskip

\noindent \textbf{Proof of Theorem \ref{thm:main2} with Lemma \ref{lemma:ineq-g}.}

 We fix $\eta>0$ to be determined later. 
By Chebychev's inequality and Lemma \ref{lemma:ineq-g}  we can find a set $K\subset B_R$ such that $\mathcal{L}^2( B_R\setminus K)\leq \eta$ and
\begin{equation}\label{ineq:cheby}
 \sup_{t\in [0,T]} \ln \left( \frac{ |X_n(t,x)-X_m(t,x)| }{\delta} +1 \right)\leq \frac{C |\ln\delta|^{2/3}}{\eta},\quad 
\text{for }\: x\in K.
\end{equation}
Using \eqref{ineq:bounded-flow}, it follows that
\begin{equation*}
 \begin{split}
  \int_{B_R} &\sup_{t\in [0,T]} |X_n(t,x)-X_m(t,x)|\,dx\\
&\leq \int_{B_R\setminus K} \sup_{t\in [0,T]} |X_n(t,x)-X_m(t,x)|\,dx+\int_{K} \sup_{t\in [0,T]} |X_n(t,x)-X_m(t,x)|\,dx\\
&\leq C\mathcal{L}^2( B_R\setminus K)+C\sup_{x\in K}\sup_{t\in [0,T]}|X_n(t,x)-X_m(t,x)|\\
&\leq C\Big(\eta +\delta \exp\big(C|\ln\delta|^{2/3}/\eta\big)\Big),
 \end{split}
\end{equation*}
where we have used \eqref{ineq:cheby} in the last inequality.
We finally optimize the choice of the parameter $\eta$ as follows. We set
\begin{equation*}
 \eta\equiv \frac{2C}{|\ln \delta|^{1/3}},
\end{equation*}
so that $\exp(C |\ln\delta|^{2/3}/\eta)=\exp(|\ln \delta|/2)=\delta^{-1/2}$. This yields
\begin{equation}\label{ineq:cauchy-2}
 \int_{B_R} \sup_{t\in [0,T]} |X_n(t,x)-X_m(t,x)|\,dx\leq \frac{C}{|\ln \delta(n,m)|^{1/3}}.
\end{equation}
 In particular, we infer that
$(X_{n})_{n\in \N}$ is a Cauchy sequence converging to some $Y:[0,T]\times \RR\to \RR$ in the space $L^1_{\loc}\left(\RR,L^\infty\left([0,T]\right)\right)$. Finally, the fact that $Y$ is the (unique) regular Lagrangian flow associated to $b$ is standard, and we omit the proof.
\hfill $\square$
\begin{remark}
Given the strong convergence of $(X_{n})_{n\in \N}$ to $X$ together with the uniform bound \eqref{ineq:liouville} we infer that under assumptions $(H'_1)-(H'_2)-(H'_3)$ the constant $L_R$ in Definition \ref{def:lagrangian-flow} actually does not depend on $R$.
\end{remark}

\medskip

We finally give the

\noindent \textbf{Proof of Lemma \ref{lemma:ineq-g}.}

Let $\eps>0$ be a small parameter to be chosen later. We consider the set
\begin{equation*}
 P(n,\eps)=\left\{x\in \RR\quad \text{s.t.}\quad \min_{t\in [0,T]}|X_n(t,x)-z(t)|<\eps\right \}. 
\end{equation*}
Using Proposition \ref{lemma:bad-set} applied to $X_n$,  with $q=\infty$,  and thanks to \eqref{ineq:liouville}, we obtain
\begin{equation}\label{ineq:measure-bad}
  \mathcal{L}^2(P(n,\eps))\leq C\eps,
\end{equation}
where $C$ depends only on $T$, $L_0$, $\|v\|_{L^\infty(L^\infty)}$, 
$\|\dot{z}\|_{L^\infty}$, and $\|\nabla v\|_{L^1(L^p)}$. Next, 
\begin{equation*}
 g_{n,m}=\mathcal{G}_{n,m}+\mathcal{B}_{n,m},
\end{equation*}
where
\begin{equation*}
 \begin{split}
\mathcal{G}_{n,m}&=\int_{B_R\setminus \left[ P(n,\eps)\cup P(m,\eps)\right]}
 \sup_{t\in [0,T]}\ln\left(\frac{| X_n(t,x)-X_m(t,x)|}{\delta}+1\right)\,dx,\\
\mathcal{B}_{n,m}&=\int_{ P(n,\eps)\cup P(m,\eps)}
 \sup_{t\in [0,T]}\ln\left(\frac{ |X_n(t,x)-X_m(t,x)|}{\delta}+1\right)\,dx.
\end{split}
\end{equation*}

By \eqref{ineq:bounded-flow} and \eqref{ineq:measure-bad},
\begin{equation}\label{ineq:bad-part}
\mathcal{B}_{n,m}\leq C|\ln \delta|\eps.
\end{equation}

\medskip

We next estimate the second part, for which we can adapt the proof of Theorem 2.9 in \cite{crippa-delellis} for Sobolev vector fields since $H$ is regular off the set $\{(t,z(t)),t\in[0,T]\}$.
 We have
\begin{equation*}
 \begin{split}
\sup_{t\in [0,T]} & \ln\left(\frac{ |X_n(t,x)-X_m(t,x)|}{\delta}+1\right)\\
&\leq \int_0^T 
\left| \frac{d}{dt} X_n(\tau,x)-\frac{d}{dt} X_m(\tau,x)\right|
\left( |X_n(\tau,x)-X_m(\tau,x)|+\delta\right)^{-1}\,d\tau\\
&\leq \int_0^T
\frac{ \left|b_{n}\left(\tau,X_n(\tau,x)\right)-b_{m}\left(\tau,X_m(\tau,x)\right)\right|}
{|X_n(\tau,x)-X_m(\tau,x)|+\delta}
\,d\tau.
\end{split}\end{equation*}
Writing $$b_n(X_n)-b_m(X_m)=\big[b_n(X_n)-b_m(X_n)\big]+\big[b_m(X_n)-b_m(X_m)\big]$$
we further obtain $\mathcal{G}_{n,m}\leq \mathcal{G}^{1}_{n,m}+\mathcal{G}^{2}_{n,m},$ where
\begin{equation*}
 \begin{split}
\mathcal{G}^{1}_{n,m} &
=\frac{1}{\delta} \int_0^T\int_{B_R}
\left|b_{n}\left(\tau,X_n(\tau,x)\right)-b_{m}\left(\tau,X_n(\tau,x)\right)\right|\,dx\,d\tau
\end{split}
\end{equation*}
and
\begin{equation*}
 \begin{split}
\mathcal{G}^{2}_{n,m} &
=  \int_0^T\int_{B_R\setminus \left[ P(n,\eps)\cup P(m,\eps)\right]}
\frac{ \left|b_{m}\left(\tau,X_n(\tau,x)\right)-b_{m}\left(\tau,X_m(\tau,x)\right)\right|}
{|X_n(\tau,x)-X_m(\tau,x)|}\,dx\,d\tau.
\end{split}
\end{equation*}

By definition of $\widetilde{R}$ and by \eqref{ineq:liouville} and \eqref{ineq:bounded-flow} we obtain
\begin{equation}\label{ineq:g1}
 \begin{split}
\mathcal{G}^{1}_{m,n}&\leq \frac{e^{L_0}}{\delta}
\int_0^T \int_{B_{\widetilde{R}}} |b_n-b_m|(\tau,y)\,dy\,d\tau= e^{L_0}.
\end{split}
\end{equation}

We now estimate $\mathcal{G}^2_{m,n}$. Let $0\leq \chi_{\eps}\leq 1$ be a smooth function such that $\chi_{\eps}=0$ on $B(0,\eps/2)$ and $\chi_{\eps}=1$ on $B(0,\eps)^c$ and let 
$$H_{m,\eps}(t,x)=(K_m\chi_\eps)(x-z(t)),\quad b_{m,\eps}=v_m+H_{m,\eps}.$$
For $x\in B_R\setminus \left[ P(n,\eps)\cup P(m,\eps)\right]$ we have $b_{m}(\tau,X_n(\tau,x))=b_{m,\eps}(\tau,X_n(\tau,x))$ and $b_{m}(\tau,X_m(\tau,x))=b_{m,\eps}(\tau,X_m(\tau,x))$ for $\tau\in[0,T]$.

In the following $M f$ denotes the maximal function of $f$. Using the classical estimate of the difference quotient of a function in terms of the maximal function of the derivative (see e.g. Lemma A.3 in \cite{crippa-delellis}) we find
\begin{equation*}\begin{split}
 \int_0^T\int_{B_R}
&\frac{ \left|b_{m,\eps}\left(\tau,X_n(\tau,x)\right)-b_{m,\eps}\left(\tau,X_m(\tau,x)\right)\right|}
{|X_n(\tau,x)-X_m(\tau,x)|}\,dx\,d\tau\\
&\leq C\int_0^T\int_{B_R} \big[M \nabla b_{m,\eps}\left(\tau, X_m(\tau,x)\right)+
M \nabla b_{m,\eps}\left(\tau, X_n(\tau,x)\right)\big]\,dx\,d\tau.
\end{split}
\end{equation*}By using \eqref{ineq:bounded-flow} and \eqref{ineq:liouville} we get
\begin{equation*}\begin{split}
 \mathcal{G}^2_{m,n}
&\leq Ce^{L_0}\int_0^T\int_{B_{\widetilde{R}}} \left|M \nabla b_{m,\eps}\left(\tau,y\right)\right|\,dy\,d\tau\\
&\leq Ce^{L_0}\widetilde{R}^{1-1/p}\int_0^T\|M \nabla b_{m,\eps}(\tau))\|_{L^p(B_{\widetilde{R}})}\,d\tau\\
&\leq  Ce^{L_0}\widetilde{R}^{1-1/p}\Big(\|\nabla v_m\|_{L^1(L^p)}+\|\nabla H_{m,\eps}\|_{L^1(L^p)}\Big).
\end{split}
\end{equation*}
In view of $(H'_2)$ and of the expression of $H_{m,\eps}$ we get
\begin{equation}\label{ineq:g2}
 \mathcal{G}^2_{m,n}\leq \frac{C}{\eps^{2-\frac{2}{p}}}\leq \frac{C}{\eps^{2}} .
\end{equation}

We gather \eqref{ineq:bad-part}, \eqref{ineq:g1} and \eqref{ineq:g2}, obtaining
\begin{equation*}
 g_{m,n}\leq C\big(\eps |\ln\delta|+\eps^{-2}\big).
\end{equation*}
We now optimize our choice of $\eps$, setting $\eps=|\ln \delta|^{-1/3}$,
so that
\begin{equation*}
 g_{m,n}\leq C |\ln\delta|^{2/3}
\end{equation*}
and the conclusion of Lemma \ref{lemma:ineq-g} follows.
\hfill $\square$

\section{Lagrangian solutions to the vortex-wave system}
\label{section:vortex-wave}

We  finally comment on the applications of the previous results to the vortex-wave system \eqref{eq:vortex-wave}. Two notions of weak solution for the vortex-wave system have been introduced: Eulerian solutions and Lagrangian solutions, see \cite{bresiliens-miot,clacave-miot}. These notions coincide when the vorticity $\omega$ belongs to $L^\infty(L^1\cap L^\infty)$ \cite{mar_pul, livrejaune, clacave-miot}. In \cite{bresiliens-miot} the authors establish global existence of an Eulerian solution with $\omega$ belonging to $L^\infty(L^1\cap L^p)$ for $p>2$. We claim that to this Eulerian solution corresponds a unique regular Lagrangian flow and that $\omega$ is constant along the flow trajectories. Indeed, if $p>2$ then  $v=\frac{1}{2\pi}K\ast \omega$ satisfies all assumptions $(H'_1)-(H'_2)-(H'_3)$, therefore also $(H_1)-(H_2)-(H_3)-(H_4)$. Hence, in view of Theorems \ref{thm:main} and \ref{thm:main2} of the present article, there exists a unique regular Lagrangian flow associated to the divergence free  velocity field $b=v+H$. Moreover, it can be readily checked (adapting, e.g., the proof of Theorem 1.3 in \cite{clacave-miot}), that the function $\widetilde{\omega}=X(t,\cdot)_\# \omega_0$ is a  distributional solution in $L^\infty(L^1\cap L^p)$ of the PDE
\begin{equation*}
\partial_t \widetilde{\omega}+(v+H)\cdot \nabla \widetilde{\omega}=0,\quad \widetilde{\omega}(0)=\omega_0.
\end{equation*}
Now, invoking the uniqueness part of Proposition \ref{prop:uniqueness} we obtain $\omega=\widetilde{\omega}$, which establishes our claim.

\medskip

Finally, we mention that Theorems \ref{thm:main} and \ref{thm:main2} can be extended to vector fields $H$ containing several point singularities
$$H(t,x)=\sum_{i=1}^N d_i K(x-z_i(t)),\quad d_i\in \R,$$
under the condition
$$\min_{i\neq j}\min_{t\in[0,T]}|z_i(t)-z_j(t)|>0,$$
 which corresponds to the interaction of several point vortices in the setting of the point vortex system.

\medskip

\textbf{Acknowledgments.} G. C. acknowledges the financial support of the SNSF grant \# 140232.
E. M. acknowledges the financial support of the PICS program of the CNRS (FLAME  project). M. C. L. F., E. M. and H. J. N. L. thank the Brazilian-French Network in Mathematics for its financial support. M. C. L. F. ackowledges the support of CNPq grant \# 303089/2010-5.
H. J. N. L. thanks the support of CNPq grant \# 306331/2010-1 and FAPERJ grant \# 103.197/2012. This work was partially supported by FAPESP grant \# 07/51490-7 and
by FAPERJ-PRONEX.

\end{document}